\documentclass[makeidx,a4paper,11pt]{article}
\usepackage{makeidx}
\usepackage{amssymb , amsmath, amsthm}
\newtheorem{defi}{Definition} 
\newtheorem{thm}[defi]{Theorem}

\newtheorem{prop}[defi]{Proposition}
\newtheorem{lemme}[defi]{Lemma}
\newtheorem{cor}[defi]{Corollary}
 
\newcommand{\twosystem}[2]{\left\{\begin{aligned} &#1\\ &#2\end{aligned}\right.}

\newcommand{\parte}[1]{\smallskip\noindent {\rm#1)}\,\,}

\newcommand{\dotp}[2]{\langle{#1},{#2}\rangle}
\newcommand{\scal}[2]{\langle{#1},{#2}\rangle}
\newcommand{\starred}[1]{{#1}^{\star}}

\newcommand{\Cal }[1]{{\mathcal {#1}}}
\newcommand{\bd}{\partial}
\newcommand{\abs}[1]{\lvert{#1}\rvert}
\newcommand{\norm}[1]{\lVert{#1}\rVert}

\newcommand{\reals}{{\bf R}}
\newcommand{\real}[1]{{\bf R}^{#1}}
\newcommand{\sphere}[1]{{\bf S}^{#1}}
\newcommand{\hyp}[1]{{\bf H}^{#1}}

\newcommand{\Vol}{{\rm Vol}}

\newcommand{\isoper}{\dfrac{{\rm Vol}(\Sigma)}{{\rm Vol}(\Omega)}}
\newcommand{\isoperi}{{\rm Vol}(\Sigma)/{\rm Vol}(\Omega)}
\newcommand{\vol}{{\rm Vol}}

\newcommand{\rough}[1]{\nabla^{\star}\nabla#1}

\newcommand{\derive}[2]{\dfrac{\partial{#1}}{\partial{#2}}}
\newcommand{\derives}[2]{\partial{#1}/\partial{#2}}

\newcommand{\ball}{\mathbf{B}^{n+1}}

\begin{document}

\title{On the first eigenvalue of the Dirichlet-to-Neumann operator on forms\footnote{Classification AMS $2000$: 58J50, 35P15 \newline 
Keywords: Manifold with boundary, Differential forms, Eigenvalue, Sharp bounds}} 
\author{ S. Raulot and A. Savo}
\date{\today}
 
\maketitle

\begin{abstract}
We study a Dirichlet-to-Neumann eigenvalue problem for differential forms on a compact Riemannian manifold with smooth boundary. This problem is  a natural generalization 
of the classical Steklov problem on functions. We derive a number of upper and lower bounds for the first eigenvalue in several 
contexts: many of these estimates will be sharp, and for some of them we characterize equality. We also relate these new eigenvalues with those of other 
operators, like the Hodge Laplacian or the biharmonic Steklov operator. 

\end{abstract}

\large²


\section{Introduction}


Let $\Omega$ be a compact, connected $(n+1)-$dimensional Riemannian domain with smooth boundary $\Sigma^n$. The Dirichlet-to-Neumann operator $T$, also 
called \it Steklov operator, \rm acts on smooth functions on $\Sigma$ in the following way. If $f\in C^{\infty}(\Sigma)$ 
and $\hat f$ denotes the unique harmonic extension of $f$ to $\Omega$, then:
\begin{equation}\label{classical}
Tf=-\derive{\hat f} N,
\end{equation}
where $N$ is the inner unit normal vector field on $\Sigma$. $T$ defines a pseudo-differential operator on $C^{\infty}(\Sigma)$ which is known to be 
elliptic and self-adjoint; hence  $T$  has a discrete spectrum  $0=\nu_1<\nu_2\leq\nu_3\leq\dots$.  Note that the lowest eigenvalue is $\nu_1=0$, 
corresponding to the constant eigenfunctions; therefore, in our convention, the first positive eigenvalue of $T$ will be denoted by $\nu_2$. There is a vast 
literature on eigenvalue estimates for the operator $T$; directly related to our paper are the estimates given in \cite{Es} and \cite{Es1}. 

\smallskip

In this paper, we consider a natural extension of the Dirichlet-to-Neumann operator $T$  to an elliptic operator $T^{[p]}$ acting on differential forms of arbitrary degree
$p$ on the boundary $\Sigma$ and then prove some geometric lower bounds for its first eigenvalue, given in terms of the second fundamental form of the 
boundary. We then estimate these new eigenvalues from above in terms of the isoperimetric ratio $\isoperi$, and in terms of the eigenvalues of other 
differential operators, like the Hodge-Laplace operator on the boundary $\Sigma$ and the biharmonic Steklov operator. In some cases we improve some known 
estimates. The operator $T^{[p]}$ seems to have interesting spectral properties which, we hope, justify the present work. 

\smallskip

In the rest of the introduction we state the main results of the paper. 


\subsection{The definition of $T^{[p]}$} \label{def}


Let $\omega$ be a form of degree $p$ on $\Sigma^n$, with $p=0,1,\dots,n$. Then there exists a unique $p$-form $\hat\omega$ on $\Omega$ such that:
$$
\twosystem
{\Delta\hat\omega=0}
{\starred J\hat\omega=\omega,\,\,i_N\hat\omega=0,}
$$
where $\starred J$ denotes the restriction of $\hat\omega$ to $\Sigma$, and $i_N$ is the interior product of $\hat\omega$ with the inner unit normal vector 
field $N$. The form $\hat\omega$ will be called the {\it harmonic tangential extension of $\omega$}. Its existence and uniqueness is proved, for example, in 
Schwarz \cite{schwarz}. We set:
$$
T^{[p]}\omega=-i_Nd\hat\omega,
$$
and then we have a linear operator $T^{[p]}:\Lambda^p(\Sigma)\to \Lambda^p(\Sigma)$, {\it the (absolute) Dirichlet-to-Neumann operator}, which reduces to the classical 
Dirichlet-to-Neumann operator acting on functions when $p=0$, so that $T^{[0]}=T$. Here $\Lambda^p(\Sigma)$ denotes the vector bundle of differential $p$-forms on 
$\Sigma$.

\smallskip

We observe in Section \ref{basic} that $T^{[p]}$ is an elliptic self-adjoint pseudo-differential operator, with discrete spectrum
$$
\nu_{1,p}(\Omega)\leq\nu_{2,p}(\Omega)\leq\dots
$$
Moreover, $T^{[p]}$ is non-negative so that $\nu_{1,p}(\Omega)\geq 0$. Actually, it follows easily from the definition that ${\rm Ker}T^{[p]}$ is isomorphic 
to $H^p(\Omega)$, the $p$-th absolute de Rham cohomology space of $\Omega$ with real coefficients. Therefore:

\medskip

$-$ a positive lower bound of $\nu_{1,p}(\Omega)$ will imply in particular  that $H^p(\Omega)=0$;

$-$ a positive upper bound of $\nu_{1,p}(\Omega)$ will be significant only when $H^p(\Omega)=0$.

\medskip

As $\Omega$ is connected, we see that $H^0(\Omega)$ is  $1-$dimensional.  Therefore, in our notation, $\nu_{1,0}(\Omega)=0$ and $\nu_{2,0}(\Omega)=\nu_2$ is 
the first positive eigenvalue of the classical problem \eqref{classical}. 

Finally, using the Hodge star operator, we define  a dual operator $T^{[p]}_D$, also acting on $\Lambda^p(\Sigma)$; in particular, the dual of  $T^{[n]}$ 
defines an operator $T^{[0]}_D$ acting on $C^{\infty}(\Sigma)$ and different from the classical Dirichlet-to-Neumann operator $T$ (see Section \ref{dual} for details).

\smallskip

The operator $T^{[p]}$ belongs to a family of operators depending on a complex parameter $z$, introduced by G. Carron in \cite{carron} (see the proof of 
Theorem \ref{FirstProp}). Other Dirichlet to Neumann operators acting on differential forms, but different from $T^{[p]}$, were introduced by Joshi and Lionheart in 
\cite{joshi}, and Belishev and Sharafutdinov in \cite{belishev}. In the preprint \cite{shara}, the operator 
$\star_{\Sigma} T^{[p]}: \Lambda^{p}(\Sigma)\to \Lambda^{n-p}(\Sigma)$ appears in a certain matrix decomposition of the Joshi and Lionheart operator. 
None of these works, however, discuss eigenvalue estimates.


\subsection{Lower bounds by the extrinsic geometry} 


First, some notations. Fix a point $x\in\Sigma$ and let $\eta_1(x),\dots,\eta_n(x)$ be the principal curvatures of $\Sigma^n$ at $x$ (our sign convention is 
that the principal curvatures of the unit ball in $\real{n+1}$ are positive). The {\it $p$-curvatures} of $\Sigma$ are, by definition, all possible sums 
$\eta_{j_1}(x)+\dots+\eta_{j_p}(x)$ for $j_1,\dots,j_p\in\{1,\dots,n\}$. Arrange the sequence of principal curvatures so that it is non-decreasing:
$\eta_1(x)\leq\dots\leq\eta_n(x)$, and call
$$
\sigma_p(x)\doteq\eta_1(x)+\dots+\eta_p(x)
$$
the \it lowest $p$-curvature at $x$. \rm We say that $\Sigma$ is \it $p$-convex \rm if $\sigma_p(x)\geq 0$ for all $x\in\Sigma$, and let
$$
\sigma_p(\Sigma)=\inf_{x\in\Sigma}\sigma_p(x).
$$
Note that $1$-convex means, simply, {\it convex} (all principal curvatures are non-negative) and $n$-convex means that $\Sigma$ has non-negative mean curvature 
because, by definition, $\sigma_n(\Sigma)=nH$, where $H$ is a lower bound of the mean curvature of $\Sigma$. Finally, it is clear from the definition 
that, if $\Sigma$ is $p-$convex, then it is $q-$convex for all $q\geq p$.

\smallskip

Recall that, if $\omega$ is a $p$-form on $\Omega^{n+1}$, the Bochner formula gives
$$
\Delta\omega=\rough\omega+W^{[p]},
$$
where $W^{[p]}$ is a symmetric endomorphism acting on $\Lambda^{p}(\Omega)$, called the {\it Bochner curvature term}. One knows that $W^{[1]}={\rm Ric}$, 
the Ricci tensor, hence $W^{[1]}\geq 0$ provided that $\Omega$ has nonnegative Ricci curvature.

From the work of Gallot and Meyer (see \cite{gallot-meyer}) we also know that, if $\gamma$ is a lower bound of the eigenvalues of the Riemann curvature 
operator (seen as a symmetric endomorphism of $\Lambda^2(\Omega)$), then $W^{[p]}\geq p(n+1-p)\gamma$. Hence

\medskip

$-$ if the curvature operator of $\Omega$ is nonnegative then $W^{[p]}\geq 0$ for all degrees $p$. 

\medskip

However, the condition $W^{[p]}\geq 0$ is sometimes much weaker than assuming the positivity of the curvature operator.
\begin{thm}\label{lowerbound} 
Let $p=1,\dots,n$. Assume that $\Omega^{n+1}$ satisfies $W^{[p]}\geq 0$ and that $\Sigma$ is strictly $p$-convex, that is $\sigma_p(\Sigma)>0$.
\item (a) If $p<\dfrac{n+1}2$ then $\nu_{1,p}(\Omega)> \dfrac{n-p+2}{n-p+1}\sigma_p(\Sigma)$.The equality never holds.

\item (b) If $p\geq\dfrac{n+1}2$ then
\begin{equation}\label{lower}
\nu_{1,p}(\Omega)\geq \frac{p+1}p\sigma_p(\Sigma),
\end{equation}
which is an equality when $\Omega$ is a ball in the Euclidean space $\real{n+1}$.
\end{thm}

{\bf Remark.} Note that under the given curvature assumptions we have in particular $H^q(\Omega)=0$ for all $q\geq p$; so, the $p$-convexity has interesting 
topological consequences. This is not new:  in \cite{wu}  it was proved by other methods that, if $\sigma_p(\Sigma)>0$  and the sectional curvatures of 
$\Omega$ are non-negative, then $\Omega$ has the homotopy type of a $CW-$complex with cells only in dimensions $\leq  p-1$. For a result in negative 
curvature we refer to \cite{savo}: in particular, if $\Omega$ is a $p-$convex domain in ${\bf H}^n$ then $H^p(\Omega)=0$ for all $q\geq p$, provided that 
$p>(n+1)/2$.

\smallskip
 
The proof of Theorem \ref{lowerbound} uses a Reilly-type formula for differential forms, proved in \cite{raulotsavo}. We characterize  the equality in \eqref{lower} in the following two 
cases: when $p=n$ and when $p>(n+1)/2$ and $\Omega$ is a Euclidean domain. Precisely:  
\begin{thm}\label{equalityone}
Assume that $\Omega$ has non-negative Ricci curvature and mean-convex boundary. Then
$$
\nu_{1,n}(\Omega)\geq (n+1)H,
$$
where $H$ is a lower bound of the mean curvature. If $n\geq 2$, equality holds if and only if $\Omega$ is a Euclidean ball. 
\end{thm}

\begin{thm}\label{equalitytwo}
If $p>\frac{n+1}2$ and $\Omega$ is a Euclidean domain, then  we have equality in \eqref{lower} if and only if $\Omega$ is a ball. 
\end{thm}

For Euclidean domains we also prove an inequality relating the first eigenvalues for consecutive degrees.  
\begin{thm}\label{consecutive} 
Let $\Omega$ be any compact domain in $\real {n+1}$, and let $\sigma_p(\Sigma)$ be a lower bound of the $p$-curvatures of $\Sigma$ (which we do not assume 
to be positive).
\item (i) For all $p=1,\dots,n$ one has
$\nu_{1,p}(\Omega)\geq \nu_{1,p-1}(\Omega)+\sigma_p(\Sigma)/p.$

\item (ii) If $\Omega$ is convex, then $\nu_{1,p}>0$ for all $p\geq 1$ and
$$\nu_{1,1}(\Omega)\leq\nu_{1,2}(\Omega)\leq\dots\leq \nu_{1,n}(\Omega).$$
\end{thm}

The inequality $(i)$ is sharp for $p>\frac{n+1}{2}$ since equality is achieved by the unit Euclidean ball. The monotonicity property in $(ii)$ is an 
immediate consequence of $(i)$, because if $\Omega$ is convex then $\sigma_p(\Sigma)\geq 0$ for all $p$. 

\smallskip

We remark that the property $(ii)$ holds also for the first eigenvalues of the Laplacian acting on $p$-forms of a convex Euclidean domain $\Omega$, for the 
absolute boundary conditions (see \cite{guerini-savo}).


\subsection{Upper bounds by the isoperimetric ratio}


It turns out that the existence of parallel forms implies that, for suitable degrees, the Dirichlet-to-Neumann eigenvalues can be bounded above by the isoperimetric 
ratio $\vol(\Sigma)/\vol(\Omega)$. Precisely, if $\Omega$ supports a non trivial parallel $p$-form, and $H^p(\Omega)=H^p_R(\Omega)=0$, then
\begin{equation}\label{isoperimetricbound}
\nu_{1,p-1}(\Omega)+\nu_{1,n-p}(\Omega)\leq\isoper.
\end{equation}
In some cases the estimate is sharp and we can characterize equality. Either one of the two cohomology assumptions can be removed if the given parallel 
form is known to be exact (respectively, co-exact): so, for example, \eqref{isoperimetricbound} holds in all degrees for all domains in Euclidean space, 
since the parallel $p-$form $dx_1\wedge\dots\wedge dx_p$ is exact and co-exact.

\smallskip

The inequality \eqref{isoperimetricbound} follows from the estimates in Section \ref{upperbounds}, which apply more generally to the ratio 
$\int_{\Sigma}\norm{\xi}^2/\int_{\Omega}\norm{\xi}^2$, where $\xi$ is a {\it harmonic field}, that is, a differential form which is closed and co-closed 
(we remark that on a manifold with nonempty boundary the vector space of harmonic fields of a given degree is infinite dimensional, and is properly contained 
in the space of harmonic forms). 

\smallskip

As the volume form of $\Omega$ is parallel we have, for all compact manifolds with boundary, the estimate:
\begin{equation}\label{mainiso}
\nu_{1,n}(\Omega)\leq\dfrac{\Vol(\Sigma)}{\Vol(\Omega)},
\end{equation} 
which reduces to an equality when $\Omega$ is a Euclidean ball. 

\smallskip

Then, we examine the equality case in \eqref{mainiso}. To that end, consider the mean-exit time function $E$, solution of the problem:
$$
\twosystem
{\Delta E=1\quad\text{on}\quad \Omega,}
{E=0\quad\text{on}\quad \Sigma.}
$$
Any domain for which the normal derivative $\derives EN$ is constant on $\Sigma$ will be called a \it harmonic domain. \rm The reason for this terminology is 
given by Proposition \ref{mefc}, in which we observe the following simple fact: $\derives EN$ is constant on $\Sigma$ if and only if the mean value of any 
harmonic function on $\Omega$ equals its mean value on the boundary.

\begin{thm}\label{equalityup} 
Let $\Omega$ be any compact domain. Then $\nu_{1,n}(\Omega)\leq \isoperi$.

\parte a If equality holds, then $\Omega$ is a harmonic domain.

\parte b Conversely, if $\Omega$ is a harmonic domain, then $\isoperi$ belongs to the spectrum of $T^{[n]}$ (an associated eigenform being $\star dE$).
\end{thm}
\smallskip

It remains to see how rigid  the harmonicity condition is, and what conditions it imposes on the geometry of the boundary. For Euclidean domains the question 
was settled in a famous paper by Serrin \cite{serrin} which states in particular that any harmonic domain in $\real{n+1}$ is a ball. This rigidity result was 
extended by Kumaresan and Prajapat (see \cite{kumaresan}) to domains in the hyperbolic space $\hyp{n+1}$ and in  the hemisphere $\sphere{n+1}_+$. To our knowledge,
the classification of harmonic domains in $\sphere{n+1}$ is still an open (and interesting) question. Then, we have the following
\begin{cor} 
\parte a  For Euclidean domains  the equality holds in (\ref{mainiso}) iff $\Omega$ is a ball.

\parte b  Let $\Omega$ be a domain in $\mathbf{H}^{n+1}$ or in $\mathbf{S}^{n+1}_+$. If the equality holds in (\ref{mainiso}), then $\Omega$ is a geodesic 
ball. 
\end{cor} 

Finally, using the estimate \eqref{mainiso} and the inequalities of Theorem \ref{lowerbound} and Theorem \ref{consecutive}, one gets the following fact.
\begin{prop}\label{ballspec}
For the unit Euclidean ball $\mathbf{B}^{n+1}$ in $\real{n+1}$ one has $\nu_{1,p}(\mathbf{B}^{n+1})=p+1$ for all $p\geq (n+1)/2$.
\end{prop}
This calculation shows that the estimates of Theorems \ref{lowerbound} and \ref{consecutive} are indeed sharp.

\medskip

{\bf Remark.}
In a forthcoming paper, we will compute the whole spectrum of the Dirichlet-to-Neumann operator acting on $p-$forms of the unit Euclidean ball. In particular it turns out 
that, if $1\leq p< (n+1)/2$, then $p+1$ is still an eigenvalue of $T^{[p]}$, however it is no longer the first. In that range one has in fact
$\nu_{1,p}(\mathbf{B}^{n+1})=\dfrac{n+3}{n+1} p$.


\subsection{Upper bounds by the Hodge-Laplace eigenvalues}  


The Hodge Laplacian acting on $p$-forms of a closed manifold $\Sigma$ is the operator defined by $\Delta^\Sigma=d^\Sigma\delta^\Sigma+\delta^\Sigma d^\Sigma$,
where $d^\Sigma$ and $\delta^\Sigma$ denote respectively the differential and the co-differential acting on forms of $\Sigma$. We let $\lambda_{1,p}'(\Sigma)$ (resp. $\lambda_{1,p}''(\Sigma)$) be the first eigenvalue of $\Delta^\Sigma$ restricted to the subspace of exact (resp. co-exact) forms 
(these subspaces are preserved by $\Delta^\Sigma$ because it commutes with $d^\Sigma$ and $\delta^\Sigma$). Differentiating eigenforms, one sees that, if 
$\lambda_{1,p}(\Sigma)$ is the first positive eigenvalue of $\Delta^\Sigma$, then 
$\lambda_{1,p}(\Sigma)=\min\{\lambda_{1,p}'(\Sigma), \lambda_{1,p+1}'(\Sigma)\}$.

\smallskip

We then have the following lower bound.
\begin{thm}\label{steklapl}
Assume that $H^p_R(\Omega)=0$, $\min(\sigma_p(\Sigma),\sigma_{n-p+1}(\Sigma))\geq 0$ and $W^{[p]}\geq 0$. Then, for all $p=1,\dots,n$:
\begin{eqnarray*}
\lambda'_{1,p}(\Sigma)\geq\frac{1}{2}\big(\sigma_p(\Sigma)\nu_{1,n-p}(\Omega)+\sigma_{n-p+1}(\Sigma)\nu_{1,p-1}(\Omega)\big).
\end{eqnarray*}
\end{thm}

Observe that $\lambda'_{1,1}(\Sigma)=\lambda_1(\Sigma)$, the first positive eigenvalue of the Laplacian acting on functions of $\Sigma$. Taking $p=1$ in the previous theorem we obtain
the following sharp lower bound.
\begin{thm}\label{esc} 
Assume that $\Omega$ has non-negative Ricci curvature and that $\Sigma$ is strictly convex, with principal curvatures bounded below by $\sigma_1(\Sigma)>0$.
Then:
\begin{eqnarray*}
\lambda_1(\Sigma)\geq \dfrac12\big(\sigma_1(\Sigma)\nu_{1,n-1}(\Omega)+nH\nu_{2,0}(\Omega)\big),
\end{eqnarray*}
where $H$ is a lower bound of the mean curvature of $\Sigma$, and $\nu_{2,0}(\Omega)$ is the first positive eigenvalue of the Dirichlet-to-Neumann operator on functions.
Moreover, if $n=\dim(\Sigma)\geq 3$, the equality holds if and only if $\Omega$ is a Euclidean ball.
\end{thm}

The motivation for looking at such a bound was given by the following estimate of Escobar \cite{Es1}, which holds under the same assumptions of Theorem \ref{esc}:
\begin{equation}\label{escobar}
\lambda_1(\Sigma)>\dfrac {nH}2\nu_{2,0}(\Omega).
\end{equation}
We observe that the defect $\lambda_1(\Sigma)-\dfrac {nH}2\nu_{2,0}(\Omega)$ in \eqref{escobar}  is bounded below by the first Dirichlet-to-Neumann eigenvalue in the degree $n-1$, thus obtaining a sharp bound.


\subsection{An upper bound by the first biharmonic Steklov eigenvalue}


The following problem on functions is classical, and is known as the fourth order (or {\it biharmonic}) Steklov eigenvalue problem:
\begin{equation}\label{biharmonic}
\twosystem
{\Delta^2f=0\quad\text{on $\Omega$},}
{f=0,\,\Delta f=\mu\frac{\partial f}{\partial N}\quad\text{on $\Sigma$}.}
\end{equation}
For recent results on the problem, we refer to \cite{ferrero} and \cite{wangxia}. An immediate application of the min-max principle associated to the 
Dirichlet-to-Neumann operator on $n$-forms gives:
\begin{thm}\label{munu}
One has always $\mu_1(\Omega)\geq \nu_{1,n}(\Omega)$, where $\mu_1(\Omega)$ is the first eigenvalue of \eqref{biharmonic}.
If the equality holds, then $\Omega$ is a harmonic domain.
\end{thm}

In \cite{wangxia} Wang and Xia  prove that, if the Ricci curvature of $\Omega$ is non-negative and the mean curvature of $\Sigma$ is bounded  below by $H>0$, 
then $\mu_1(\Omega)\geq (n+1)H$. Moreover equality occurs if and only if $\Omega$ is isometric to a ball of $\reals^{n+1}$.
Combining Theorem \ref{munu} and our estimate of Theorem \ref{equalityone} we see that, under the given assumptions:
$$
\mu_1(\Omega)\geq \nu_{1,n}(\Omega)\geq (n+1)H
$$
which implies the result of Wang and Xia. 
On the other hand, it is easy to observe that 
$\mu_1(\Omega)\leq\vol(\Sigma)/\vol(\Omega)$ (see for example \cite{wangxia}). Then the estimate (\ref{mainiso}) is a direct consequence of this fact and 
Theorem \ref{munu}.

\smallskip

The paper is organized as follows. In Section \ref{basic} we state the main properties of the operator $T^{[p]}$. In Section
\ref{lb} we prove the lower bounds and in Section \ref{upperbounds} we give the proof of the upper bounds. Finally, in Section \ref{appendix},  we prove a rigidity result needed for the equality case of Theorem \ref{equalitytwo}.


\section{Generalities on the Dirichlet-to-Neumann operator} \label{basic}


Before stating the main properties of $T^{[p]}$, let us recall the following well-known facts. The Hodge-de Rham theorem for 
manifolds with boundary asserts that $H^p_{\rm dR}(\Omega,\reals)$, the absolute de Rham cohomology space in degree $p$ with real 
coefficients, is isomorphic to the (finite dimensional) vector space of  harmonic $p$-forms $\phi$ satisfying the absolute boundary conditions ($i_N\phi=i_Nd\phi=0$ on $\Sigma$), which we  denote
by $H^p(\Omega)$.  Equivalently, one has:
$$
H^p(\Omega)=\{\phi\in\Lambda^p(\Omega): d\phi=\delta\phi=0 \,\,\text{on}\,\,\Omega, \, i_N\phi=0\,\,\text{on}\,\,\Sigma\}.
$$
By duality, the relative de Rham cohomology space in degree $p$ is isomorphic to the vector space
$$
H^p_R(\Omega)=\{\phi\in\Lambda^p(\Omega): d\phi=\delta\phi=0 \,\,\text{on}\,\,\Omega, \, \starred J\phi=0\,\,\text{on}\,\,\Sigma\}.
$$

\begin{thm}\label{FirstProp} 
Let $\Omega^{n+1}$ be a compact domain with smooth boundary $\Sigma^n$. Let $T^{[p]}$ be the Dirichlet-to-Neumann operator acting on $p-$forms of  $\Sigma$, as defined in Section \ref{def}. Then:
\item (a)  $T^{[p]}$ is nonnegative and self-adjoint.      

\item (b) The kernel of $T^{[p]}$ consists of the boundary values of absolute cohomology classes, and  the restriction 
$\starred J$ induces an isomorphism between $H^p(\Omega)$ and ${\rm Ker}(T^{[p]})$.

\item (c) $T^{[p]}$ is an elliptic pseudo-differential operator of order one. Hence  it admits an increasing sequence of eigenvalues with finite multiplicities
$$
\nu_{1,p}(\Omega)\leq\nu_{2,p}(\Omega)\leq\dots
$$
with $\nu_{1,p}(\Omega)=0$ repeated $b_p(\Omega)={\rm dim}\,H^p(\Omega)$ times. In particular, $\nu_{1,p}(\Omega)>0$ if and only if $H^p(\Omega)=0$.

\item (d) The first  eigenvalue of $T^{[p]}$ satisfies the  min-max principle
\begin{eqnarray}\label{varchar}
\nu_{1,p}(\Omega)={\inf}\left\{\dfrac
{\int_{\Omega}\norm{d\hat\phi}^2+\norm{\delta\hat\phi}^2}
{\int_{\Sigma}\norm{\hat\phi}^2}\right\}
\end{eqnarray}
where the infimum is taken over all $p-$forms $\hat\phi$ on $\Omega$ such that $i_N\hat\phi=0$ on $\Sigma$.
\end{thm}

We remark that $(b)$ has already been  observed  in \cite{shara}.

\begin{proof}

\smallskip

{\it  $(a)$} We prove that the operator is self-adjoint. 
Recall the Stokes formula:
\begin{eqnarray*}
\int_{\Omega}\dotp{d\omega_1}{\omega_2}=\int_{\Omega}\dotp{\omega_1}{\delta\omega_2}
-\int_{\Sigma}\dotp{\starred J\omega_1}{i_N\omega_2}
\end{eqnarray*}
for all $\omega_1\in\Lambda^{p-1}(\Omega)$ and $\omega_2\in\Lambda^p(\Omega)$. Now let $\phi,\psi\in\Lambda^p(\Sigma)$ and denote by $\hat\phi$, $\hat\psi$ their harmonic 
tangential extensions on $\Omega$. The definition of $T^{[p]}$ and the Stokes formula give:
$$
\int_{\Sigma}\dotp{T^{[p]}\phi}{\psi} =  -\int_{\Sigma}\dotp{i_Nd\hat\phi}{\starred J\hat\psi}= \int_{\Omega}\dotp{d\hat\phi}{d\hat\psi}-\dotp{\delta d\hat\phi}{\hat\psi}.
$$
As $\hat\psi$ is harmonic and $i_N\hat\psi=0$ we have
$$
-\int_{\Omega}\dotp{\delta d\hat\phi}{\hat\psi}=\int_{\Omega}
\dotp{d \delta\hat\phi}{\hat\psi}=\int_{\Omega}\dotp{\delta \hat\phi}{\delta\hat\psi}.
$$
So $\int_{\Sigma}\dotp{T^{[p]}\phi}{\psi}=\int_{\Omega}\dotp{d\hat\phi}{d\hat\psi}+\dotp{\delta \hat\phi}{\delta\hat\psi}$
which shows that  $T^{[p]}$ is self-adjoint. Taking $\psi=\phi$  yields:
\begin{eqnarray*}
\int_{\Sigma}\dotp{T^{[p]}\phi}{\phi}=\int_{\Omega}\norm{d\hat\phi}^2+\norm{\delta\hat\phi}^2\geq 0.
\end{eqnarray*}
and $T^{[p]}$ is nonnegative.  

\smallskip

\it $(b)$ \rm If
$\phi\in {\rm Ker}(T^{[p]})$ then its harmonic tangential extension $\hat\phi$ satisfies, on $\Sigma$:
$i_N\hat\phi= i_N d\hat\phi=0$.
Hence $\phi$ is the restriction of a form (cohomology class) in $H^p(\Omega)$. Conversely, it is clear by the definition that an absolute cohomology class 
restricts to a form in the kernel of $T^{[p]}$. Then:
$$
{\rm Ker}(T^{[p]})=\starred J(H^p(\Omega)).
$$
We observe that the map $\starred J: H^p(\Omega)\to\starred J\big(H^p(\Omega)\big)$ is injective: in fact, if 
$\starred J\hat\phi=0$ for some cohomology class $\hat\phi$, then $\hat\phi$ is harmonic and zero on the boundary, which 
implies $\hat\phi=0$. Then the dimension of ${\rm Ker}(T^{[p]})$ equals $b_p(\Omega)$. 

\smallskip

 {\it $(c)$}  The proof that $T^{[p]}$ is an elliptic pseudo-differential operator follows the lines of the proof done in Section 6.4 of \cite{carron}. 
There, in studying determinants,  G. Carron considers the linear operator $T_z:\Lambda^p(\Sigma)\to \Lambda^p(\Sigma)$ depending on a complex parameter $z\in{\bf C}\setminus [0,\infty)$ and 
 defined by
$$
T_z\phi=-i_Nd\hat\phi_z,
$$
where $\hat\phi_z$ is the unique solution of 
\begin{equation}\label{z}
\twosystem
{\Delta\hat\phi_z=z\hat\phi_z\quad\text{on}\quad \Omega,}
{i_N\hat\phi_z=0, \starred J\hat\phi_z=\phi \quad\text{on}\quad \Sigma.}
\end{equation}
Carron shows that $T_z$ is an elliptic, pseudo-differential, invertible operator. In fact, the inverse $S_z$ of $T_z$ is shown to be the operator obtained by 
restricting to the boundary the Green kernel of the Hodge Laplacian $\Delta$ acting on $p$-forms of $\Omega$, for the absolute boundary conditions; as $S_z$ 
is pseudo-differential of order $-1$, the operator $T_z$ is pseudo-differential of order $1$.  The restriction on $z$ is imposed precisely because then $T_z$ 
will be  invertible, since $z$ avoids the spectrum of $\Delta$ (which is contained in the nonnegative half-line).

\smallskip

Our operator is obtained by taking $z=0$ in \eqref{z}: it is no longer invertible when $H^p(\Omega)\ne \{0\}$  but it is still pseudo-differential and elliptic because, 
by $(b)$, its kernel is finite dimensional, isomorphic to $H^p(\Omega)$. In fact, the operator $S_0$ is now invertible modulo compact operators, given by the 
projection onto the kernel of $T_0$ and its transpose. The rest of Carron's proof carries over and so $T_0=T^{[p]}$ is an elliptic PDO of order $1$. More 
generally, $T_z$ is an elliptic PDO for all $z$, and is invertible as long as $z$ does not belong to the spectrum of $\Delta$.

\smallskip

The rest of $(c)$ now follows from the standard theory of elliptic PDO (see \cite{shubin}).

\smallskip

{\it $(d)$} \rm The min-max principle gives
$$
\nu_{1,p}(\Omega)=\inf\left\{\dfrac{\int_{\Omega}\norm{d\hat\phi}^2+\norm{\delta\hat\phi}^2}
{\int_{\Sigma}\norm{\hat\phi}^2}: \Delta\hat\phi=0, i_N\hat\phi=0\right\}.
$$
We only have to show that we can remove the condition $\Delta\hat\phi=0$. This follows from the fact that among all tangential 
extensions $\xi$ of a given form $\phi\in\Lambda^p(\Sigma)$, the harmonic tangential extension $\hat\phi$ minimizes the quadratic form
$\int_{\Omega}\norm{d\xi}^2+\norm{\delta\xi}^2$. Indeed, assume that $\starred J\xi=\phi=\starred J\hat\phi$ and $i_N\xi=0=i_N\hat\phi$.  
Let  $\psi=\xi-\hat\phi$ so that $\psi=0$ on the boundary. Using the Stokes formula one verifies that:
$$
0\leq\int_{\Omega}\norm{d\psi}^2+\norm{\delta\psi}^2=\int_{\Omega}\norm{d\xi}^2+\norm{\delta\xi}^2-\int_{\Omega}\norm{d\hat\phi}^2+\norm{\delta\hat\phi}^2,
$$
and the assertion  follows. 
\end{proof}

\medskip

 
\subsection{The dual problem}\label{dual}


Let $p=0,\dots,n$. Given a $p$-form $\phi$ on $\Sigma$ consider the unique $(p+1)$-form $\tilde\phi$ on $\Omega$ which satisfies:
$$
\twosystem
{\Delta\tilde\phi=0\quad\text{on}\quad \Omega}
{\starred J\tilde\phi=0,\,\,i_N\tilde\phi=\phi\quad\text{on}\quad \Sigma.}
$$
The form $\tilde\phi$ will be called the {\it harmonic normal extension} of $\phi$. Its existence and uniqueness is also
proved in Schwarz \cite{schwarz}. We set 
\begin{eqnarray*}
T^{[p]}_D\phi=\starred J(\delta\tilde\phi)
\end{eqnarray*}
and call $T^{[p]}_D$ {\it the relative Dirichlet-to-Neumann operator}. 
It defines another elliptic pseudo-differential operator of order one acting on $\Lambda^{p}(\Sigma)$, which is self-adjoint and nonnegative. These properties can easily 
be derived from Theorem \ref{FirstProp} and the fact that $T^{[p]}_D$ is related to the absolute Dirichlet-to-Neumann operator by the identity 
$T^{[p]}_D=(-1)^{p(n-p)}\star_\Sigma T^{[n-p]}\star_\Sigma$, where $\star_\Sigma$ denotes the Hodge-star operator acting on forms on $\Sigma$. Denoting by 
$\nu^D_{1,p}(\Omega)$ the first eigenvalue of $T^{[p]}_D$, we have
$$
\nu^D_{1,p}(\Omega)=\nu_{1,n-p}(\Omega).
$$
Moreover, the min-max principle for the dual problem takes the form:
\begin{eqnarray}\label{dircar}
\nu_{1,p}^D(\Omega)={\inf}\left\{\dfrac
{\int_{\Omega}\norm{d\hat\phi}^2+\norm{\delta\hat\phi}^2}
{\int_{\Sigma}\norm{\hat\phi}^2}:\hat\phi\in\Lambda^{p+1}(\Omega),\,\starred J\hat\phi=0\right\}.
\end{eqnarray}

Note that $T^{[0]}_D$ is an operator acting on functions, which clearly differs from the operator $T^{[0]}$.


\section{Lower bounds: proofs}\label{lb}



\subsection{Reilly formula for differential forms}


The main tool used in the proof of the lower bound is a  Reilly-type formula for differential forms 
proved by the authors in \cite{raulotsavo}, which we state below. 

\smallskip

Denote by $S$  the shape operator of the immersion of $\Sigma$ in $\Omega$; it is defined as 
$S(X)=-\nabla_XN$ for all tangent vectors $X\in T\Sigma$. $S$ admits a canonical extension acting on $p$-forms on $\Sigma$ and 
denoted by $S^{[p]}$. Explicitly, if $\omega$ is a $p$-form on $\Sigma$ one has:
\begin{eqnarray*}
S^{[p]}\omega(X_1,\dots,X_p)=\sum_{j=1}^p \omega(X_1,\dots,S(X_j),\dots,X_p),
\end{eqnarray*}
for tangent vectors $X_1,\dots,X_p\in T\Sigma$. It is clear from the definition that the eigenvalues of $S^{[p]}$ are precisely 
the $p$-curvatures of $\Sigma$: therefore we have immediately
\begin{eqnarray*}
\dotp{S^{[p]}\omega}{\omega}\geq \sigma_p(\Sigma)\norm{\omega}^2
\end{eqnarray*}
at all points of $\Sigma$ and for all $p$-forms $\omega$. Now let $\omega$ be a $p-$form on $\Omega$. The Reilly formula says that 
\begin{eqnarray}\label{r}
\int_{\Omega}\norm{d\omega}^2+\norm{\delta\omega}^2=
\int_{\Omega}\norm{\nabla\omega}^2+\scal{W^{[p]}(\omega)}{\omega}
+2\int_{\Sigma}\scal{i_N\omega}{\delta^{\Sigma}(J^{\star}\omega)}
+\int_{\Sigma}{\Cal B}(\omega,\omega),
\end{eqnarray}
where the boundary term has the following expression: 
$$
\begin{aligned}
{\Cal B}(\omega,\omega)&=\scal{S^{[p]}(J^{\star}\omega)}{J^{\star}\omega}
+nH\norm{i_N\omega}^2-\scal{S^{[p-1]}(i_N\omega)}{i_N\omega}\\
&=
\scal{S^{[p]}(J^{\star}\omega)}{J^{\star}\omega}
+\scal{S^{[n-p+1]}(J^{\star}\star\omega)}{J^{\star}\star\omega}
\end{aligned}
$$
By convention, we set $S^{[0]}=S^{[n+1]}=0$.
For a detailed proof of \eqref{r} see \cite{raulotsavo}.


\subsection{Proof of Theorem \ref{lowerbound}}


We assume that $W^{[p]}\geq 0$, and that the $p$-curvatures of $\Sigma$ are bounded below by $\sigma_p(\Sigma)>0$. We have to prove that, if
$p<\frac{n+1}2$ then:
\begin{equation}\label{boundone}
\nu_{1,p}(\Omega)>\frac{n-p+2}{n-p+1}\sigma_p(\Sigma),
\end{equation}
and if $p\geq\frac{n+1}2$ then
\begin{equation}\label{boundtwo}
\nu_{1,p}(\Omega)\geq \frac{p+1}{p}\sigma_p(\Sigma).
\end{equation}

Let $\omega$ be an eigenform associated to $\nu_{1,p}(\Omega)$ and let $\hat\omega$ be its harmonic tangential extension to $\Omega$. By the variational 
characterization (\ref{varchar}):
\begin{equation}\label{proofone}
\int_{\Omega}\norm{d\hat\omega}^2+\norm{\delta\hat\omega}^2=\nu_{1,p}(\Omega)\int_{\Sigma}\norm{\omega}^2
\end{equation}
because, on the boundary, $\norm{\hat\omega}^2=\norm{\omega}^2+\norm{i_N\hat\omega}^2=\norm{\omega}^2$.
We apply the Reilly formula to $\hat\omega$. As $W^{[p]}\geq 0$ and $i_N\hat\omega=0$ we get
\begin{equation}\label{prooftwo}
\begin{aligned}
\int_{\Omega}\Big(\norm{d\hat\omega}^2+\norm{\delta\hat\omega}^2\Big)&\geq \int_{\Omega}\norm{\nabla\hat\omega}^2
+\int_{\Sigma}\dotp{S^{[p]}(\omega)}{\omega}\\
&\geq \int_{\Omega}\norm{\nabla\hat\omega}^2
+\sigma_{p}(\Sigma)\int_{\Sigma}\norm{\omega}^2
\end{aligned}
\end{equation}

We will use the following estimate of Gallot and Meyer \cite{gallot-meyer}, valid  for any $p$-form $\hat\omega$:
\begin{eqnarray}\label{twist}
\norm{\nabla\hat\omega}^2\geq \dfrac{\norm{d\hat\omega}^2}{p+1}+\dfrac{\norm{\delta\hat\omega}^2}{n-p+2}.
\end{eqnarray}
When $p<\frac{n+1}{2}$ one has $p+1< n-p+2$ hence:
\begin{equation}\label{proofthree}
\norm{\nabla\hat\omega}^2\geq \dfrac{\norm{d\hat\omega}^2+\norm{\delta\hat\omega}^2}{n-p+2},
\end{equation}
and the equality implies $d\hat\omega=0$. Inserting \eqref{proofthree} in \eqref{prooftwo}, and taking into account \eqref{proofone}, we obtain 
\eqref{boundone}. Note that then $\nu_{1,p}(\Omega)>0$. Equality in \eqref{proofthree} implies that $d\hat\omega=0$ hence $i_Nd\hat\omega=0$: but this is impossible because 
otherwise $\nu_{1,p}(\Omega)=0$. So the inequality is always strict.

\smallskip

If $p\geq \frac{n+1}{2}$ one has
\begin{equation}\label{prooffour}
\norm{\nabla\hat\omega}^2\geq \dfrac{\norm{d\hat\omega}^2+\norm{\delta\hat\omega}^2}{p+1}
\end{equation}
and proceeding as before we obtain \eqref{boundtwo}. The inequality \eqref{boundtwo} is sharp: for the unit Euclidean ball we have $\sigma_p(\Sigma)=p$ and $\nu_{1,p}(\ball)=p+1$ (see Proposition \ref{ballspec}).
We finally remark that, if $p>\frac{n+1}2$ and the equality holds in \eqref{boundtwo}, it holds also in \eqref{prooffour} and then $\delta\hat\omega=0$.
\hfill$\square$\\

Now we study the equality case of this estimate. Recall that the $p$-form $\hat\omega$ is a {\it conformal Killing form} if it satisfies the differential equation
\begin{eqnarray*}
\nabla_X\hat\omega=\frac{1}{p+1}i_Xd\hat\omega-\frac{1}{n-p+2}X^*\wedge\delta\hat\omega
\end{eqnarray*}
for all $X\in T\Omega$.
A co-closed conformal Killing form is called a {\it Killing form}. It is well-known that the inequality \eqref{twist} is an equality if and only if 
$\hat\omega$ is a conformal Killing form (see for example \cite{gallot-meyer}). We then have:
\begin{prop}\label{equalitycase}
Assume $p\geq (n+1)/2$. If equality holds in (\ref{boundtwo}) then the harmonic tangential extension of a $p$-eigenform 
associated to $\nu_{1,p}(\Omega)$ is a conformal Killing $p$-form (a Killing form if $p>\frac{n+1}{2}$) and the $p$ lowest principal 
curvatures of the boundary are constant, equal to $c={\nu_{1,p}(\Omega)}/(p+1)$.  
\end{prop}

\begin{proof}
Looking at the proof of \eqref{boundtwo} we see immediately that if the equality holds then $\hat\omega$ is a conformal Killing form and, by the last remark 
in the proof, it is a Killing form when $p>\frac{n+1}2$. It remains to show the last assertion. Now, the Gauss formula leads to the following relations 
(see Section 6 in \cite{raulotsavo}):
\begin{equation}\label{derivative}
\twosystem
{\nabla_X^{\Sigma}(i_N\hat\omega)=i_N\nabla_X\hat\omega-i_{S(X)}\starred J\hat\omega}
{\nabla_X^{\Sigma}(\starred J\hat\omega)=\starred J(\nabla_X\hat\omega)+S(X)^{\star}\wedge i_N\hat\omega,}
\end{equation}
for all $X\in T\Sigma$,  where $\nabla^{\Sigma}$ is the Levi-Civita connection of $\Sigma$. Since $\hat\omega$ is the 
harmonic tangential extension of $\omega$, we have $i_N\hat\omega=0$ and the first equation in (\ref{derivative}) reads:  
\begin{eqnarray}\label{normal}
i_N\nabla_X\hat\omega=i_{S(X)}\omega.
\end{eqnarray}
On the other hand, since $\hat\omega$ is a conformal Killing $p$-form we have for all $X\in\Gamma(T\Sigma)$:
\begin{equation}\label{normalone}
 i_N\nabla_X\hat\omega = -\frac{1}{p+1}i_X(i_Nd\hat\omega)=\frac{\nu_{1,p}(\Omega)}{p+1}i_X\omega.
\end{equation}
We used the fact that $i_N\delta\hat\omega=-\delta^\Sigma(i_N\hat\omega)=0$, which immediately implies $i_N(\starred X\wedge\delta\hat\omega)=0$.
Combining (\ref{normal}) and (\ref{normalone}) gives:
\begin{equation}\label{normaltwo}
i_{S(X)-\frac{\nu_{1,p}(\Omega)}{p+1}X}\omega=0
\end{equation}
for all $X\in\Gamma(T\Sigma)$. The form $\omega$, being an eigenform of an elliptic operator, can't vanish on an open set and therefore is non-zero a.e. on 
$\Sigma$. Take a point $x$ where it does not vanish: then, at $x$, there exists $p$ principal directions, say $v_{1},\dots,v_{p}$, such that 
$\omega(v_{1},\dots,v_{p})\ne 0$. Choosing successively $X=v_{1}, \dots,v_p$ one sees from \eqref{normaltwo} that the associated principal curvatures
satisfy $\lambda_1=\dots=\lambda_p=\frac{\nu_{1,p}(\Omega)}{p+1}$. 
\end{proof}


\subsection{Proof of Theorem \ref{equalityone}}


Assume that $\Omega^{n+1}$ has nonnegative Ricci curvature and that $\Sigma$ has mean curvature bounded below by $H>0$. Then $\sigma_n=nH$ and applying Theorem \ref{lowerbound} for $p=n$ we get
$\nu_{1,n}(\Omega)\geq (n+1)H$. It remains to show that, if the equality holds, then $\Omega$ is a Euclidean ball. Now, under the given assumptions, we have $\vol(\Sigma)/\vol(\Omega)\geq (n+1)H$ by Theorem 1 in \cite{ros}, with equality if and only if $\Omega$ is a Euclidean ball. 
It is then enough to show that
$$
\isoper=(n+1)H.
$$
From Proposition \ref{equalitycase}, we know that if $\omega\in\Lambda^n(\Sigma)$ is a eigenform associated with $\nu_{1,n}(\Omega)=(n+1)H$, then its 
harmonic tangential extension $\hat\omega\in\Lambda^n(\Omega)$ is a Killing $n$-form on $\Omega$; in particular, $\delta\hat\omega=0$. We can write 
$d\hat\omega=f\Psi_{\Omega}$, where $\Psi_{\Omega}$ is the volume form of $\Omega$ and $f$ is a smooth function. As $\hat\omega$ is harmonic and co-closed,
 we have
$$
0=\delta d\hat\omega=\delta (f\Psi_{\Omega})=-i_{\nabla f}\Psi_{\Omega},
$$
which immediately implies $\nabla f=0$. By renormalization, we can assume that $f=1$ and so $d\hat\omega$ is the volume form of $\Omega$. 
By assumption,  $\starred Jd\hat\omega=0$ and $i_Nd\hat\omega=-(n+1)H\omega$. Then,  on $\Sigma$
$$
1=\norm{i_Nd\hat\omega}^2=(n+1)^2H^2\norm{\omega}^2.
$$
On the other hand, by the Stokes formula and the fact that $d\hat\omega$ has constant unit norm:
$$
\vol(\Omega)=\int_{\Omega}\norm{d\hat\omega}^2=-\int_{\Sigma}
\dotp{\omega}{i_Nd\hat\omega}=(n+1)H\int_{\Sigma}\norm{\omega}^2=
\dfrac{\vol(\Sigma)}{(n+1)H},
$$
which proves the assertion.


\subsection{The equality case for Euclidean domains: proof of Theorem \ref{equalitytwo}}\label{equalityeuclidean}


We fix $c>0$ and let ${\cal F}_p(c)$ denote the set of $p$-forms $\hat\omega$ on $\Omega$, $p=0,\dots,n$, with the following properties:

\parte a $\hat\omega$ is  harmonic and tangential (that is $i_N\hat\omega=0$ on $\Sigma$).

\parte b $\hat\omega$ is Killing and $d\hat\omega$ is parallel.

\parte c $i_Nd\hat\omega=-(p+1)c\omega$, where $\omega=\starred J\hat\omega$ is the restriction to $\Sigma$.

\smallskip

Note that ${\cal F}_0(c)$ consists of all harmonic functions $\hat f$ with parallel gradient and such that $\frac{\bd\hat  f}{\bd N}=-c\hat f$: if $\hat f$ 
is not trivial, its restriction to the boundary is a Dirichlet-to-Neumann eigenfunction associated to the eigenvalue $c$.

\smallskip

\begin{lemme}\label{Fpc}
Let $p\geq 1$. If $\hat\omega\in{\cal F}_p(c)$ and $V$ is a parallel vector field on $\real{n+1}$, then $i_V\hat\omega\in{\cal F}_{p-1}(c)$.
\end{lemme}

\begin{proof}
 The Cartan formula gives
$
di_V\hat\omega+i_Vd\hat\omega={\cal L}_V\hat\omega,
$
where ${\cal L}_V$ is the Lie derivative along $V$. If $V$ is parallel and $\hat\omega$ is Killing, we have 
${\cal L}_V\hat\omega=\nabla_V\hat\omega=\frac1{p+1}i_Vd\hat\omega$ and then:
\begin{equation}\label{cartan}
di_V\hat\omega=-\frac{p}{p+1}i_Vd\hat\omega.
\end{equation}
Now $\nabla_Vd\hat\omega={\cal L}_Vd\hat\omega=di_Vd\hat\omega=0$ by Cartan formula and \eqref{cartan}. This holds for all parallel vector fields: in particular, any Killing
form of degree $p\geq 1$ in Euclidean space has parallel exterior derivative.

\smallskip

Fix $\hat\omega\in{\cal F}_p(c)$. As $V$ is parallel, $i_V$ commutes with $\Delta$ and anticommutes with $i_N$. Then $i_V\hat\omega$ satisfies a). 
 
\smallskip
  
As $i_V$ anticommutes with $\delta$, we see that $i_V\hat\omega$ is co-closed. On the other hand, since $V$ is parallel:
$$
\nabla_Xi_V\hat\omega=i_V\nabla_X\hat\omega=
\frac{1}{p+1}i_Vi_Xd\hat\omega=
-\frac{1}{p+1}i_Xi_Vd\hat\omega=\frac1pi_Xdi_V\hat\omega
$$
where we used \eqref{cartan} in the last equality. Hence $i_V\hat\omega$ is a Killing $(p-1)-$form. A similar calculation shows that $\nabla_Xdi_V\hat\omega=0$, 
hence $di_V\hat\omega$ is parallel and b) follows.

\smallskip

Finally, again using \eqref{cartan}: 
$$
i_Ndi_V\hat\omega=-\frac{p}{p+1}i_Ni_Vd\hat\omega=
\frac{p}{p+1}i_Vi_Nd\hat\omega=-pci_V\hat\omega,
$$
and c) follows as well. 
\end{proof}

Now assume that $\Omega$ is an extremal domain for our inequality, and let $\hat\omega$ be the tangential harmonic extension of an eigenform $\omega$ associated to $\nu_{1,p}(\Omega)$. Set $c=\nu_{1,p}(\Omega)/(p+1)$.
By Proposition \ref{equalitycase}, $\hat\omega$ is a Killing $p-$form: in particular, as observed in the proof of the Lemma \ref{Fpc}, $d\hat\omega$ is parallel. Moreover, $i_Nd\hat\omega=-(p+1)c\omega$ by definition. This means that $\hat\omega$ is a form in ${\cal F}_p(c)$. As $\hat\omega$ is non 
trivial, we can find $p$ parallel vector fields $V_1,\dots,V_p$ such that the function
$
\hat f=\hat\omega(V_1,\dots,V_p)
$
is non trivial. Applying  the lemma successively to the parallel fields $V_1,\dots,V_p$, we see that $\hat f\in {\cal F}_0(c)$, that is, $\hat f$ satisfies
$$
\twosystem
{\nabla d\hat f=0\quad\text{on $\Omega$},}
{\derive{\hat f}{N}=-c\hat f \quad\text{on $\Sigma$}.}
$$
By Proposition \ref{equalitycase},  the lowest $p$ principal curvatures are constant, equal to $c$, and then $S\geq c$. We now apply Theorem \ref{obata} in the Appendix, to conclude that $\Omega$ is a Euclidean ball. The proof of 
Theorem \ref{equalitytwo} is now complete.


\subsection{An inequality for consecutive degrees: proof of Theorem \ref{consecutive}}


We have to show that if $\Omega$ is a domain in $\reals^{n+1}$, then for all $p=1,\dots,n$:
\begin{equation}\label{consecutive1}
\nu_{1,p}(\Omega)\geq \nu_{1,p-1}(\Omega)+\frac{\sigma_p(\Sigma)}{p}.
\end{equation}

For the proof, we consider the family of unit length parallel vector fields on $\real{n+1}$, which is naturally identified with $\mathbf{S}^n$. 

\smallskip

Let $\omega\in\Lambda^p(\Sigma)$ be an eigenform associated to the eigenvalue $\nu_{1,p}(\Omega)$ and denote by $\hat\omega$ its harmonic tangential extension. 
Let $V$ be a unit length parallel vector field. Since $\Delta$ commutes with the contraction $i_V$, the $(p-1)$-form $i_{{V}}\hat\omega$ is harmonic. 
Moreover we clearly have $i_Ni_V\hat\omega=0$. Hence we can use $i_V\hat\omega$  as test form for the eigenvalue $\nu_{1,p-1}(\Omega)$, and by the min-max 
principle we have 
\begin{eqnarray}\label{zero}
\nu_{1,p-1}(\Omega)\int_{\Sigma}\norm{i_V\hat\omega}^2
\leq\int_{\Omega}\norm{di_V\hat\omega}^2+\norm{\delta i_V\hat\omega}^2
\end{eqnarray}
for all $V\in\sphere n$. Now we want to integrate this inequality with respect to $V\in\mathbf{S}^n$. In order to simplify the formulae, we use the 
renormalized measure 
$$
d\mu=\dfrac{n+1}{\vol(\sphere n)}\,{\rm dvol}_{\sphere n},
$$
where ${\rm dvol}_{\sphere n}$ is the canonical measure of $\sphere n$. Then, we have the following identities, which are valid pointwise and are proved in 
\cite{guerini-savo} (Lemma $4.8$, p. $336$):
$$
\begin{aligned}
\int_{\mathbf{S}^n}\norm{i_V\hat\omega}^2 d\mu(V) & = p\norm{\hat\omega}^2\\
\int_{\mathbf{S}^n}\norm{di_V\hat\omega}^2 d\mu(V) & =\norm{\nabla\hat\omega}^2+(p-1)\norm{d\hat\omega}^2\\
\int_{\mathbf{S}^n}\norm{\delta i_V\hat\omega}^2 d\mu(V) & =\int_{\mathbf{S}^n}
\norm{i_V\delta\hat\omega}^2 d\mu(V)= (p-1)\norm{\delta\hat\omega}^2.
\end{aligned}
$$
Integrating (\ref{zero}) with respect to $V\in\mathbf{S}^n$ and using the previous identities, we then have, by the Fubini theorem: 
$$
p\nu_{1,p-1}(\Omega)\int_{\Sigma}\norm{\hat\omega}^2
\leq\int_{\Omega}\norm{\nabla\hat\omega}^2+(p-1)\int_{\Omega}\norm{d\hat\omega}^2+\norm{\delta\hat\omega}^2.
$$
On the other hand, the Reilly formula (\ref{r}) applied to $\hat\omega$ gives:
$$
\begin{aligned}
\int_{\Omega}\norm{d\hat\omega}^2+\norm{\delta\hat\omega}^2=\int_{\Omega}\norm{\nabla\hat\omega}^2+
\int_{\Sigma}\langle S^{[p]}(\starred J\hat\omega,)\starred J\hat\omega\rangle\geq \int_{\Omega}\norm{\nabla\hat\omega}^2+
\sigma_p(\Sigma)\int_{\Sigma}\norm{\hat\omega}^2.
\end{aligned}
$$
Eliminating $\int_{\Omega}\norm{\nabla\hat\omega}^2$ in the previous two inequalities leads to:
$$
\begin{aligned}
p\nu_{1,p-1}(\Omega)\int_{\Sigma}\norm{\hat\omega}^2
&\leq p\int_{\Omega}\big(\norm{d\hat\omega}^2+\norm{\delta\hat\omega}^2\big)-\sigma_p(\Sigma)\int_{\Sigma}\norm{\hat\omega}^2\\
&=p\nu_{1,p}(\Omega)\int_{\Sigma}\norm{\hat\omega}^2-
\sigma_p(\Sigma)\int_{\Sigma}\norm{\hat\omega}^2
\end{aligned}
$$
Dividing both sides by $p\int_{\Sigma}\norm{\hat\omega}^2$ proves (\ref{consecutive1}).


\section{Upper Bounds: proofs}\label{upperbounds}



\subsection{Upper bounds by the isoperimetric ratio}\label{isoratio}


A $p$-form $\xi$ is said to be a {\it harmonic field} if  $d\xi=\delta\xi=0$. We start from the following:
\begin{prop}\label{L2}
Let $\xi$ be a harmonic field of degree $p$ on $\Omega$. 

\item (a) If $\xi$ is exact and $p=2,\dots,n+1$ then
$\nu_{1,p-1}(\Omega)\int_{\Omega}\norm{\xi}^2\leq \int_{\Sigma}\norm{i_N\xi}^2$.

\item (b) If $\xi$ is exact and $p=1$ then $\nu_{2,0}(\Omega)\int_{\Omega}\norm{\xi}^2\leq \int_{\Sigma}\norm{i_N\xi}^2$.

\item (c) If $\xi$ is co-exact and $p=1,\dots,n$ then $\nu_{1,n-p}(\Omega)\int_{\Omega}\norm{\xi}^2\leq \int_{\Sigma}\norm{\starred J\xi}^2$.
\end{prop}

\begin{proof} 

\smallskip

(a) By the Hodge-Morrey decomposition (see \cite{schwarz}) if $\xi$ is an exact $p-$form, there is a unique co-exact (hence co-closed)  $(p-1)-$form $\omega$, called the {\it canonical primitive of $\xi$}, which satisfies:
$$
\twosystem
{d\omega=\xi}
{i_N\omega=0\quad\text{on $\Sigma$.}}
$$
We use $\omega$ as a test-form for the eigenvalue $\nu_{1,p-1}(\Omega)$ and then
$$
\nu_{1,p-1}(\Omega)\int_{\Sigma}\norm{\omega}^2
\leq \int_{\Omega}\norm{d\omega}^2.
$$
By the Stokes formula $\int_{\Omega}\norm{d\omega}^2=-\int_{\Sigma}\dotp{i_Nd\omega}{\starred J\omega}$; by the Schwarz inequality
$$
\Big(\int_{\Omega}\norm{d\omega}^2\Big)^2
\leq\int_{\Sigma}\norm{i_Nd\omega}^2\cdot\int_{\Sigma}\norm{\omega}^2.
$$
Eliminating $\int_{\Sigma}\norm{\omega}^2$ from the previous two inequalities we get
$$
\nu_{1,p-1}(\Omega)\int_{\Omega}\norm{d\omega}^2
\leq \int_{\Sigma}\norm{i_Nd\omega}^2,
$$
which is the assertion. We remark that the equality holds if and only if the canonical primitive of $\xi$ is an eigenform of $T^{[p-1]}$ associated to 
$\nu_{1,p-1}(\Omega)$.

\smallskip

(b) If $\xi$ is an exact harmonic field of degree $1$, then $\xi=df$ for an harmonic function $f$. We can assume that $f$ integrates to zero on $\Sigma$, and 
so we can use $f$ as a test function for the eigenvalue $\nu_{2,0}(\Omega)$. The rest of the proof is as in (a).

\smallskip

(c) Let $\xi$ be a co-exact $p-$harmonic field. Then $\star\xi$ is an exact $(n-p+1)-$harmonic field and we can apply (a) to it. The inequality follows 
because $\norm{i_N\star\xi}^2=\norm{\star_\Sigma\starred J\xi}^2=\norm{\starred J\xi}^2$. If the equality holds, then the canonical primitive of $\star\xi$ is an 
eigenform of $T^{[n-p]}$ associated to $\nu_{1,n-p}(\Omega)$. 

\smallskip

We can also characterize the equality by duality, as follows. If $\xi$ is co-exact, it has a unique {\it canonical co-primitive}, that is, a unique exact
$(p+1)$-form $\alpha$ such that:
$$
\twosystem{\delta\alpha=\xi}
{\starred J\alpha=0\quad\text{on}\quad\Sigma.}
$$
It is clear that if we have equality then $\alpha$  is an eigenform of the dual operator $T^{[p]}_D$ associated to $\nu_{1,p}^D(\Omega)=\nu_{1,n-p}(\Omega)$. That is,
$$
\starred J(\delta\alpha)=\nu_{1,n-p}(\Omega)i_N\alpha.
$$
\end{proof}
We remark that if $H^p(\Omega)=0$ (resp. $H^p_R(\Omega)=0$) then any $p-$harmonic field is automatically exact (resp. co-exact). Therefore, as at any point 
of the boundary one has $\norm{\xi}^2=\norm{\starred J\xi}^2+\norm{i_N\xi}^2$, we have, summing the two inequalities of the Proposition:
\begin{cor}\label{corL2}
Assume that $H^p(\Omega)=H^p_R(\Omega)=0$. Let $\xi$ be a harmonic field of degree $p$.

\item (a) If $p=2,\dots,n$ then 
$\nu_{1,p-1}(\Omega)+\nu_{1,n-p}(\Omega)\leq\int_{\Sigma}\norm{\xi}^2/\int_{\Omega}\norm{\xi}^2$. If $\xi$ is parallel then it has constant norm and
$$
\nu_{1,p-1}(\Omega)+\nu_{1,n-p}(\Omega)\leq\isoper.
$$
\item (b) If $p=1$ then 
$\nu_{2,0}(\Omega)+\nu_{1,n-1}(\Omega)\leq\int_{\Sigma}\norm{\xi}^2/\int_{\Omega}\norm{\xi}^2$.

\item (c) In particular, if $H^1_R(\Omega)=0$ and $f$ is any harmonic function then
$$
\nu_{2,0}(\Omega)+\nu_{1,n-1}(\Omega)\leq\dfrac{\int_{\Sigma}\norm{df}^2}
{\int_{\Omega}\norm{df}^2}.
$$
\end{cor}

On the other hand, the volume form of $\Omega$ is parallel, exact and has degree $n+1$. Then it follows directly from the first point of the Proposition 
\ref{L2} that, for all compact manifolds  with boundary, one has the sharp bound:
\begin{equation}\label{universal}
\nu_{1,n}(\Omega)\leq\isoper.
\end{equation}

We have equality in \eqref{universal} when $\Omega=\ball$ is the unit Euclidean ball: in fact $\vol(\Sigma)/\vol(\Omega)=n+1$ and by the main lower bound (Theorem \ref{lower}) we have $\nu_{1,n}(\ball)\geq n+1$. So 
$\nu_{1,n}(\ball)=n+1$.
We will reprove \eqref{universal} and discuss its equality case in Section \ref{HD}. 

\smallskip

We end this section with the following calculation.
\begin{prop} 
We have  $\nu_{1,p}(\ball)=p+1$ for all $p\geq\frac{n+1}2$.
\end{prop}

\begin{proof} Let $\Omega=\ball$ and let $\nu_{1,p}=\nu_{1,p}(\ball)$. We just observed that $\nu_{1,n}=n+1$. We now use Theorem \ref{consecutive}; 
as $\sigma_p(\Sigma)/p=1$ for all $p$, we see that $\nu_{1,p}\geq \nu_{1,p-1}+1$. Then $\nu_{1,n-1}\leq n$ and, by induction, $\nu_{1,p}\leq p+1$ for all $p$. 
However, when $p\geq (n+1)/2$, Theorem \ref{lowerbound} applied to $\Omega$ gives $\nu_{1,p}\geq p+1$ and so $\nu_{1,p}=p+1$.
\end{proof}

For later use, we observe the following 
\begin{prop}\label{linfunc}
Assume that $\Omega$ supports a non constant linear function, that is, a smooth function $f$ with $df$ non trivial and parallel. If $H^1_R(\Omega)=0$, then:
$$
\nu_{2,0}(\Omega)+\nu_{1,n-1}(\Omega)\leq \isoper.
$$

\item (a) If the equality holds, then $\Sigma$ has constant positive mean curvature $H=\nu_{1,n-1}(\Omega)/n$, and the restriction of $f$ to $\Sigma$ is an 
eigenfunction of $\Delta^{\Sigma}$ associated to the eigenvalue $\lambda\doteq \nu_{2,0}(\Omega)\nu_{1,n-1}(\Omega)$. 

\item (b) If $n=\dim(\Sigma)\geq 3$ and $\Omega\subset\real{n+1}$, then the equality holds if and only if $\Omega$ is a ball.
\end{prop}

\begin{proof}
  The inequality follows immediately from $(c)$ of Corollary \ref{corL2} applied to $df$ (which has constant norm by our assumptions). We can assume that $f$ integrates to zero on $\Sigma$. 
  
  $(a)$ If the equality holds, then $f$ has to be a Dirichlet-to-Neumann 
eigenfunction associated to $\nu_{2,0}(\Omega)$:
$$
\derive fN=-\nu_{2,0}(\Omega)f,
$$
and (see the proof of Proposition \ref{L2}) the canonical co-primitive $\alpha$ of $df$, solution of the problem
$$
\twosystem
{\delta\alpha=df,\,\,d\alpha=0}
{J^\star\alpha=0\quad\text{on $\Sigma$}}
$$
must be a dual eigenform associated to $\nu_{1,1}^D(\Omega)=\nu_{1,n-1}(\Omega)$: 
$$
\starred J(\delta\alpha)=\nu_{1,n-1}(\Omega) i_N\alpha.
$$
As  $\starred J(\delta\alpha)=d^{\Sigma}f$ we have
$
d^{\Sigma}f=\nu_{1,n-1}(\Omega)i_N\alpha.
$
It follows that
$$
\begin{aligned}
\Delta^{\Sigma}f&=\delta^{\Sigma}d^{\Sigma}f=\nu_{1,n-1}(\Omega)\delta^{\Sigma}i_N\alpha=-\nu_{1,n-1}(\Omega)i_N\delta \alpha\\
&=-\nu_{1,n-1}(\Omega)\derive fN=\nu_{2,0}(\Omega)\nu_{1,n-1}(\Omega)f,
\end{aligned}
$$
that is, $f$ is an eigenfunction of $\Delta^{\Sigma}$ associated to $\nu_{2,0}(\Omega)\nu_{1,n-1}(\Omega)$, as asserted. Observe that then
$\nu_{1,n-1}(\Omega)>0$ otherwise $f$ would be constant. To prove the first assertion, recall that, for any smooth function on $\Omega$ one has, at all points of 
$\Sigma$:
$$
\Delta f=\Delta^{\Sigma}f-\frac{\bd^2 f}{\bd N^2}+nH\derive fN.
$$
As $\nabla^2f=0$, we have $\Delta f=0$ and $\frac{\bd^2 f}{\bd N^2}=0$, and we easily obtain $nH=\nu_{1,n-1}(\Omega)$.

\smallskip

$(b)$ The equality holds for the Euclidean unit ball, by Proposition \ref{ballspec} (it is known that $\nu_{2,0}({\bf B}^{n+1})=1$).
 Now, if the equality holds, then $\Sigma$ has constant mean curvature by $(a)$, hence 
$\Sigma$ is a sphere by a well-known result of Alexandrov. 
\end{proof}


\subsection{Harmonic domains}\label{HD} 


Recall that the domain $\Omega$ is called \it harmonic \rm if $\derives EN$ is constant on $\Sigma$, where $E$ is the mean-exit time function, solution of the problem $\Delta E=1$ on $\Omega, E=0$ on $\Sigma$. Any ball in a constant curvature space form is harmonic, simply because the mean-exit time function is radially symmetric. We observe the following equivalent condition.

\begin{prop}\label{mefc}
$\Omega$ is harmonic if and only if,  for all harmonic functions $f$ on $\Omega$,  one has:
$$
\dfrac 1{\vol (\Omega)}\int_{\Omega}f=\dfrac1{\vol (\Sigma)}\int_{\Sigma}f
$$
(that is, the mean value of any harmonic function on the domain equals its mean value on the boundary). 
\end{prop}
\begin{proof} 
Assume that $\Omega$ is harmonic and let $f$ be any harmonic  function on $\Omega$. By the definition of $E$ and the Green formula, we have:
$$
\int_{\Omega}f=\int_{\Omega}f\Delta E=\int_{\Sigma}f\derive{E}{N}.
$$
As $\derives EN$ is constant, say equal to $c$, we have  $\int_{\Omega}f=c\int_{\Sigma}f$. Taking $f=1$ we see that 
$c=\vol(\Omega)/\vol(\Sigma)$ and the 
first half is proved. 

\smallskip

Conversely, assume that the above mean-value property is true for all harmonic functions on $\Omega$ . Fix a point $x\in\Sigma$ and let $f_k\in C^{\infty}(\Sigma)$ be a sequence of  functions converging to 
the Dirac measure of $\Sigma$ at $x$ as $k\to\infty$. Let $\hat f_k$ be the harmonic extension of $f_k$. Then
$
\int_{\Omega}\hat f_k=\int_{\Sigma}f_k\derives{E}{N}
$
and the assumption gives
$$
\dfrac{\vol(\Omega)}{\vol(\Sigma)}\int_{\Sigma}f_k=\int_{\Sigma}f_k\derive{E}{N}
$$
for all $k$. Letting  $k\to\infty$ we obtain
$$
\dfrac{\vol(\Omega)}{\vol(\Sigma)}=\derive{E}{N}(x).
$$
As $x$ is arbitrary, we see that $\derives EN$ is indeed constant on $\Sigma$.
\end{proof}


\subsection{Proof of Theorem \ref{equalityup}}


It is perhaps simpler to reprove the inequality using the dual operator $T^{[0]}_D$, with first eigenvalue $\nu_{1,0}^D(\Omega)=\nu_{1,n}(\Omega)$. So, we need to show that $\nu_{1,0}^D(\Omega)\leq\isoperi$.
Consider the $1-$form $\alpha=dE$. Then $\starred J\alpha=0$ and we can use $\alpha$ as a test-form for $\nu_{1,0}^D(\Omega)$. Since  $i_N\alpha=\derives EN$,  by the variational characterization \eqref{dircar} we get
$$
\nu_{1,0}^D(\Omega)\int_{\Sigma}\left(\derive EN\right)^2\leq\int_{\Omega}\norm{\delta\alpha}^2=\vol(\Omega).
$$
By the Schwarz inequality:
$$
\int_{\Sigma}\left(\derive EN\right)^2\geq\dfrac 1{\vol(\Sigma)}
\left(\int_{\Sigma}\derive EN\right)^2=\dfrac{\vol(\Omega)^2}{\vol(\Sigma)}
$$
and the inequality follows immediately.

\smallskip
If the equality holds then $\derives EN$ must be constant and then $\Omega$ is a harmonic domain.
Conversely,  assume that $\Omega$ is harmonic. Then  the normal derivative of $E$ is constant  along $\Sigma$,  and equals $c=\vol(\Omega)/\vol(\Sigma)$. Let $\alpha=dE$. Then
$$
\twosystem
{\Delta\alpha=0}
{\starred J\alpha=0,\,\,i_N\alpha=\derive EN=c.}
$$
By the definition of $T^{[0]}_D$:
$$
T_D^{[0]}(c)=\starred J(\delta\alpha)=1
$$
because $\delta\alpha=\Delta E=1$. This shows that $1/c$ is an eigenvalue of $T_D^{[0]}$ as asserted, and the associated 
eigenfunction is constant.


\subsection{Hodge-Laplace eigenvalues: proof of Theorem \ref{steklapl}}

           
Fix a degree $p=1,\dots,n$. We assume that $H^p_R(\Omega)=0$, $\min(\sigma_p (\Sigma),\sigma_{n-p+1}(\Sigma))\geq 0$ 
and $W^{[p]}\geq 0$.  We have to show:
\begin{eqnarray}\label{hodge}
\lambda'_{1,p}(\Sigma)\geq\frac{1}{2}\big(\sigma_p(\Sigma)\nu_{1,n-p}(\Omega)+\sigma_{n-p+1}(\Sigma)\nu_{1,p-1}(\Omega)\big).
\end{eqnarray}

Let $\phi$ be a co-exact eigenform associated to $\lambda=\lambda_{1,p-1}''(\Sigma)=\lambda_{1,p}'(\Sigma)$ and consider the exact $p$-eigenform 
$\omega=d^\Sigma\phi$ also associated to $\lambda$. Let  $\hat\phi$ be a solution of
$$
\twosystem
{\Delta\hat\phi=0\quad\text{on}\quad \Omega,}
{\starred J\hat\phi=\phi, \starred J(\delta\hat\phi)=0\quad\text{on}\quad \Sigma,}
$$
which exists by Lemma 3.4.7 in \cite{schwarz}. Then, using the Stokes formula one checks that $\delta d\hat\phi=0$ on $\Omega$
(the extension $\hat\phi$ first appeared  in the paper of  Duff and Spencer  \cite {duff}). 

If we let $\hat\omega=d\hat\phi$, 
then $\hat\omega$ is an exact $p$-harmonic field satisfying:
$$
\twosystem
{d\hat\omega=\delta\hat\omega=0\quad\text{on $\Omega$}}
{J^{\star}\hat\omega=\omega\quad\text{on $\Sigma$}.}
$$
We apply the Reilly formula (\ref{r}) to $\hat\omega$; as $W^{[p]}\geq 0$ and 
$\delta^{\Sigma}(\starred J\hat\omega)=\delta^{\Sigma}d^{\Sigma}\phi=\lambda\phi$ we obtain 
\begin{eqnarray*}
-2\lambda\int_{\Sigma}\scal{i_N\hat\omega}{\phi}
\geq
\int_{\Sigma}\dotp{S^{[p]}(\starred J\hat\omega)}{\starred J\hat\omega}+\dotp{S^{[n-p+1]}(\starred J\star\hat\omega)}{\starred J\star\hat\omega}.
\end{eqnarray*}
The Stokes formula gives:
\begin{eqnarray*}
\int_{\Sigma}\scal{i_N\hat\omega}{\phi}=\int_{\Sigma}\scal{i_Nd\hat\phi}{J^\star\hat\phi}= \int_{\Omega}\scal{\hat\phi}{\delta d\hat\phi}
-\int_{\Omega}\norm{d\hat\phi}^2=-\int_{\Omega}\norm{\hat\omega}^2.
\end{eqnarray*}
 By our curvature assumptions, we end-up with
\begin{eqnarray*}
2\lambda\int_{\Omega}\norm{\hat\omega}^2\geq \sigma_p(\Sigma)\int_{\Sigma}\norm{J^\star\hat\omega}^2+
\sigma_{n-p+1}(\Sigma)\int_{\Sigma}\norm{i_N\hat\omega}^2.
\end{eqnarray*}
The $p$-harmonic field $\hat\omega$ is exact, and also co-exact because $H^p_R(\Omega)=0$. We can then apply Proposition \ref{L2} to estimate the boundary 
integrals in the right hand side, and the estimate \eqref{hodge} follows.


\subsection{Proof of Theorem \ref{esc}}


Let $\lambda_1(\Sigma)$ be the first positive eigenvalue of the Laplacian on functions of $\Sigma$. We assume that $\Omega$ has nonnegative Ricci curvature and that $\Sigma$ is strictly convex, with principal curvatures bounded below by $\sigma_1(\Sigma)>0$. We have to show that
\begin{equation}\label{esc1}
\lambda_1(\Sigma)\geq \dfrac12\left(\sigma_1(\Sigma)\nu_{1,n-1}(\Omega)+nH\nu_{2,0}(\Omega)\right).
\end{equation}
Moreover, if $n=\dim (\Sigma)\geq 3$, the equality holds if and only if $\Omega$ is a Euclidean ball.

\begin{proof}

Let $\phi$ be an eigenfunction associated to $\lambda_1(\Sigma)$, $\hat\phi$ its harmonic extension to $\Omega$ and $\hat\omega=d\hat\phi$. Then $\hat\omega$ is an 
harmonic field of degree $1$. We apply the Reilly formula to $\hat\omega$; as 
$\norm{\nabla\hat\omega}^2\geq 0$ and $\sigma_n(\Sigma)=nH$, we obtain:
\begin{eqnarray*}
2\lambda\int_{\Omega}\norm{\hat\omega}^2\geq \sigma_1(\Sigma)\int_{\Sigma}\norm{J^\star\hat\omega}^2+
nH\int_{\Sigma}\norm{i_N\hat\omega}^2.
\end{eqnarray*}
Note that, if the equality holds, then $\hat\omega$ must be parallel. Our curvature assumptions imply in particular that $H^1_R(\Omega)=0$.
 Therefore we can apply Proposition \ref{L2} and obtain
\begin{eqnarray}\label{eq}
\int_{\Sigma}\norm{\starred J\hat\omega}^2\geq \nu_{1,n-1}(\Omega)\int_{\Omega}
\norm{\hat\omega}^2\quad\text{and}\quad
\int_{\Sigma}\norm{i_N\hat\omega}^2\geq \nu_{2,0}(\Omega)\int_{\Omega}
\norm{\hat\omega}^2.
\end{eqnarray}
The lower bound \eqref{esc1} follows. 
The estimate is sharp because, for the Euclidean unit ball, we have $\lambda_1(\Sigma)=\lambda_1(\mathbf{S}^n)=n$, $\nu_{2,0}(\ball)=1$ and, for $n\geq 3$,  $\nu_{1,n-1}(\ball)=n$. 

\smallskip

Now assume that \eqref{esc1} is an equality. Then $\hat\omega=d\hat\phi$ is parallel, and we can apply Proposition \ref{linfunc} to $f=\hat\phi$.
However, as we must have equalities in (\ref{eq}), we conclude that 
$$
\isoper =\nu_{2,0}(\Omega)+\nu_{1,n-1}(\Omega),
$$
and we are in the equality case of Proposition \ref{linfunc}. So the mean curvature is constant: $nH=\nu_{1,n-1}(\Omega)$ 
and $\lambda_1(\Sigma)=\nu_{2,0}(\Omega)\nu_{n-1}(\Omega)$. By assumption 
$$
2\lambda_1(\Sigma)=\sigma_1(\Sigma)\nu_{1,n-1}(\Omega)+nH\nu_{2,0}(\Omega)
$$
and we easily obtain $\nu_{2,0}(\Omega)=\sigma_1(\Sigma)$. Now, at each point of $\Sigma$, the mean curvature is always no less than the lowest principal curvature, which  implies that $H\geq\sigma_1(\Sigma)=\nu_{2,0}(\Omega)$.
We arrive at the inequality
$$
\isoper \leq (n+1)H.
$$
By  the result of Ros already cited (\cite{ros}) we know that $\isoperi\geq (n+1)H$ with equality if and if $\Omega$ is a Euclidean ball. Then $\Omega$ must be a Euclidean ball, and the proof is complete. \end{proof}


\subsection {Biharmonic operator: proof of Theorem \ref{munu}}\label{RFSE}


We now consider the fourth order Steklov problem \eqref{biharmonic} and its first eigenvalue $\mu_1(\Omega)$. 
As $\nu_{1,n}(\Omega)=\nu_{1,0}^D(\Omega)$ it is enough to show that
$$
\mu_1(\Omega)\geq\nu_{1,0}^D(\Omega).
$$
Let $f$ be a first eigenfunction associated to $\mu_1(\Omega)$. As $\starred J(df)=0$ we can use $df$ as a test-form in (\ref{dircar}). Then
\begin{eqnarray*}
\nu_{1,0}^D(\Omega)\leq\frac{\int_{\Omega}(\Delta f)^2}{\int_{\Sigma}\Big(\frac{\partial f}{\partial N}\Big)^2}=
\mu_1(\Omega)
\end{eqnarray*}
where the equality follows from the Rayleigh-Ritz characterization of $\mu_1(\Omega)$ (see \cite{ferrero}). 
If equality holds, then $df$ must be an eigenform of $T^{[0]}_D$ associated to $\nu_{1,0}^D(\Omega)$, hence $\Delta df=0$. But then $\Delta f$ is a constant, 
and we can assume $\Delta f=1$. As $f=0$ on $\Sigma$ we see that $f=E$, the mean-exit time function, and the boundary conditions satisfied by $f$ imply that the normal derivative of $E$ 
is constant. Hence $\Omega$ is harmonic.


\section{Appendix}\label{appendix}


Here we state a general result which gives sufficient conditions on a manifold to be isometric with a Euclidean ball. This result 
is used in the proof of Theorem \ref{equalitytwo} but it is perhaps of independent interest. 
\begin{thm}\label{obata}
Let $(\Omega^{n+1},g)$ a compact, connected Riemannian manifold with smooth boundary $\Sigma$. Assume that 
there exist a non-trivial function $f\in C^{\infty}(\Omega)$ and a number  $c>0$ such that: 
$$
\twosystem
{\nabla df=0\quad\text{on}\;\;\Omega}
{\frac{\partial f}{\partial N}=-cf\quad\text{on}\;\;\Sigma.}
$$
If $\Omega$ has non-negative sectional curvature and the second fundamental form of $\Sigma$ satisfies $S\geq c$,  
then $\Omega$ is isometric with a Euclidean ball. 
\end{thm}

\begin{proof}
It is enough to prove that the boundary is isometric to a round sphere. Then, by Theorem $1$ in \cite{xia}, we conclude that $(\Omega^{n+1},g)$ is isometric 
with a Euclidean ball.\\ Here are the main steps. We prove that:

\parte a $\Sigma$ is connected.

\parte b ${\rm Ric}^\Sigma\geq (n-1)c^2$.

\parte c $\Sigma$ has diameter greater than or equal to $\frac{\pi}{c}$.

\medskip

The proof of the Theorem will follow by observing that, by Myers' theorem and a), b), one has
${\rm diam}(\Sigma)\leq \frac{\pi}{c}$; hence, by c), the diameter is equal to $\frac{\pi}{c}$. By the rigidity theorem of Cheng \cite{cheng}, 
$\Sigma$ is isometric to a sphere of radius $1/c$, as asserted.

We prove a). Looking at the long exact sequence of the pair $(\Omega,\Sigma)$, it is enough to show that 
$H^1_R(\Omega)=0$: in fact, in that case $H^0(\Sigma)\sim H^0(\Omega)\sim\reals$. Now the Ricci curvature of $\Omega$ is non-negative
and the mean curvature of $\Sigma$ is bounded below by $c>0$: by Theorem
\ref{lower} we have $\nu_{1,n}(\Omega)>0$ and then $H^n(\Omega)=H^1_R(\Omega)=0$. 
\smallskip

We prove b). It is enough to prove that, for any unit length tangent vector $X\in T\Sigma$, one has ${\rm Ric}^\Sigma(X,X)\geq (n-1)c^2$. The Gauss lemma and 
the non-negativity of the sectional curvatures of $\Omega$ give:
$$
{\rm Ric}^\Sigma(X,X)\geq nH\scal{S(X)}{X}-\abs{S(X)}^2.
$$
Fix an orthonormal frame $(e_1,\dots,e_n)$ of principal directions, so that $S(e_j)=\eta_je_j$ for all $j$. Then:
$$
{\rm Ric}^\Sigma(X,X)\geq \sum_{j=1}^n(\eta_j(nH-\eta_j))\scal{X}{e_j}^2;
$$
as $\eta_j\geq c$ for all $j$ one sees that $\eta_j(nH-\eta_j)\geq (n-1)c^2$ for all $j$ and the assertion follows.

\smallskip

We finally prove c). Since $\nabla f$ is parallel we have that $\abs{\nabla f}$ is constant on $\Omega$, and we can assume that it is equal to $1$.
The restriction of $f$ is continuous on $\Sigma$, which is compact: then let $p_+\in\Sigma$ (resp. $p_-\in\Sigma$) be a point where the restriction of $f$ 
is maximum (resp. minimum). We prove $d(p_-,p_+)\geq \frac{\pi}{c}$. Now:
\begin{eqnarray*}
 1 =  |\nabla f|^2(p_\pm) & = & |\nabla^\Sigma f|^2(p_\pm)+\Big(\frac{\partial f}{\partial N}\Big)^2(p_\pm)\\
 & = & c^2f(p_\pm)^2.
\end{eqnarray*}
The function  $f$ is not constant on $\Sigma$ (because it is harmonic on $\Omega$ and $c>0$) therefore:
$$
f(p_+)=\frac{1}{c}, \quad f(p_-)=-\frac{1}{c}.
$$
As $\Sigma$ is connected, there exists a minimizing geodesic $\gamma:[0,l]\rightarrow \Sigma$
parametrized by arc length and joining $p_-$ with $p_+$. So we have $\gamma(0)=p_-$, $\gamma(l)=p_+$ and the distance from 
$p_-$ to $p_+$ is $l$. It is now enough to prove that $l\geq\frac{\pi}{c}$.

\smallskip

Let $\alpha(t):=f\circ\gamma(t)$ for $t\in[0,l]$, so that $\alpha'(t)\leq |\nabla^\Sigma f(\gamma(t))|$. Since 
$\nabla f$ has unit length we have:
$1=|\nabla^\Sigma f|^2\big(\gamma(t)\big)+c^2\alpha(t)^2$
and therefore
$$
|\alpha'(t)|^2\leq 1-c^2\alpha(t)^2.
$$ 
Fix $\varepsilon>0$ and let $A=\{t\in [0,l]:\alpha'(t)>0\}$. Then:
\begin{eqnarray*}
l\geq \int_Adt\geq \int_A\frac{\alpha'(t)dt}{\sqrt{1-c^2\alpha(t)^2}+
\varepsilon}\geq\int_0^l\dfrac{\alpha'(t)dt}{\sqrt{1-c^2\alpha(t)^2}+\varepsilon}.
\end{eqnarray*}
Changing variables and observing that $\alpha(0)=-\frac{1}{c}$ and $\alpha(l)=\frac{1}{c}$ we have
\begin{eqnarray*}
l\geq\frac{1}{c}\int_{-1}^1\frac{dx}{\sqrt{1-x^2}+\varepsilon}.
\end{eqnarray*}
Letting $\varepsilon\rightarrow 0^+$ gives $l\geq\frac{\pi}{c}$, as asserted.
\end{proof}

Finally, we remark that the conclusion of the Theorem holds also if the assumption on the non-negativity  of the sectional curvature is replaced by the following assumptions: the Ricci curvature of $\Omega$ is non-negative, and the mean curvature of $\Sigma$ is constant. We omit the details.



\vspace{0.8cm}     
Authors addresses:     
\nopagebreak     
\vspace{5mm}\\     
\parskip0ex     
\vtop{\hsize=6cm\noindent\obeylines}     
\vtop{     
\hsize=8cm\noindent     
\obeylines     
Simon Raulot
Laboratoire de Math\'ematiques R. Salem
UMR $6085$ CNRS-Universit\'e de Rouen
Avenue de l'Universit\'e, BP.$12$
Technop\^ole du Madrillet
$76801$ Saint-\'Etienne-du-Rouvray, France}     
     
\vspace{0.5cm}     
     
E-Mail:     
{\tt simon.raulot@univ-rouen.fr }  

\vtop{\hsize=6cm\noindent\obeylines}     
\vtop{     
\hsize=9cm\noindent     
\obeylines     
Alessandro Savo
Dipartimento SBAI, Sezione di Matematica 
Sapienza Universit\`a di Roma
Via Antonio Scarpa 16
 00161 Roma, Italy         
}     
     
\vspace{0.5cm}     
     
E-Mail:     
{\tt savo@dmmm.uniroma1.it  } 



\end{document}